\documentclass[11pt]{preprint}
\usepackage[full]{textcomp}
\usepackage[osf]{newtxtext} 
\usepackage{colortbl}

\usepackage{comment}
\usepackage[usenames,dvipsnames]{xcolor}

\usepackage{amssymb}
\usepackage{lmodern}
\usepackage{mathtools}

\usepackage{hyperref}
\usepackage{breakurl}
\usepackage{aligned-overset}
\usepackage{mhenvs}
\usepackage{mhequ} 
\newcommand{\be}{\begin{equation*}}
\newcommand{\ee}{\end{equation*}}
\usepackage{mhsymb}
\usepackage{booktabs}
\usepackage{tikz}
\usepackage{tcolorbox}
\usepackage{mathrsfs}
\usepackage[utf8]{inputenc}
\usepackage{longtable}
\usepackage{wrapfig}
\usepackage{subcaption}
\usepackage{mathrsfs}
\usepackage{epsfig}
\usepackage{microtype}
\usepackage{comment}
\usepackage{wasysym}
\usepackage{centernot}
\usepackage{enumitem}
\usepackage{bm}
\usepackage{stackrel}
\usepackage{graphicx}
\makeatletter
\newcommand{\globalcolor}[1]{%
  \color{#1}\global\let\default@color\current@color
}
\makeatother

\usetikzlibrary{calc}
\usetikzlibrary{decorations}
\usetikzlibrary{positioning}
\usetikzlibrary{shapes}
\usetikzlibrary{external}

\definecolor{blush}{rgb}{0.87, 0.36, 0.51}
	\definecolor{brightcerulean}{rgb}{0.11, 0.67, 0.84}
	\definecolor{greenryb}{rgb}{0.4, 0.69, 0.2}

\newif\ifdark
\darkfalse

\ifdark
\definecolor{darkred}{rgb}{0.9,0.2,0.2}
\definecolor{darkblue}{rgb}{0.7,0.3,1}
\definecolor{darkgreen}{rgb}{0.1,0.9,0.1}
\definecolor{franck}{rgb}{0,0.8,1}
\definecolor{pagebackground}{rgb}{.15,.21,.18}
\definecolor{pageforeground}{rgb}{.84,.84,.85}
\pagecolor{pagebackground}
\AtBeginDocument{\globalcolor{pageforeground}}
\definecolor{symbols}{rgb}{0,0.7,1}
\colorlet{connection}{red!80!black}
\colorlet{boxcolor}{blue!50}

\else

\definecolor{darkred}{rgb}{0.7,0.1,0.1}
\definecolor{darkblue}{rgb}{0.4,0.1,0.8}
\definecolor{darkgreen}{rgb}{0.1,0.7,0.1}
\definecolor{franck}{rgb}{0,0,1}
\definecolor{pagebackground}{rgb}{1,1,1}
\definecolor{pageforeground}{rgb}{0,0,0}
\colorlet{symbols}{blue!90!black}
\colorlet{connection}{red!30!black}
\colorlet{boxcolor}{blue!50!black}

\fi

\def\slash{\leavevmode\unskip\kern0.18em/\penalty\exhyphenpenalty\kern0.18em}
\def\dash{\leavevmode\unskip\kern0.18em--\penalty\exhyphenpenalty\kern0.18em}

\DeclareMathAlphabet{\mathbbm}{U}{bbm}{m}{n}

\DeclareFontFamily{U}{BOONDOX-calo}{\skewchar\font=45 }
\DeclareFontShape{U}{BOONDOX-calo}{m}{n}{
  <-> s*[1.05] BOONDOX-r-calo}{}
\DeclareFontShape{U}{BOONDOX-calo}{b}{n}{
  <-> s*[1.05] BOONDOX-b-calo}{}
\DeclareMathAlphabet{\mcb}{U}{BOONDOX-calo}{m}{n}
\SetMathAlphabet{\mcb}{bold}{U}{BOONDOX-calo}{b}{n}

\setlist{noitemsep,topsep=4pt,leftmargin=1.5em}

\DeclareMathAlphabet{\mathbbm}{U}{bbm}{m}{n}

\DeclareMathAlphabet{\mcb}{U}{BOONDOX-calo}{m}{n}
\SetMathAlphabet{\mcb}{bold}{U}{BOONDOX-calo}{b}{n}
\DeclareFontFamily{U}{mathx}{\hyphenchar\font45}
\DeclareFontShape{U}{mathx}{m}{n}{
      <5> <6> <7> <8> <9> <10>
      <10.95> <12> <14.4> <17.28> <20.74> <24.88>
      mathx10
      }{}
\DeclareSymbolFont{mathx}{U}{mathx}{m}{n}
\DeclareMathSymbol{\bigtimes}{1}{mathx}{"91}

\def\s{\mathfrak{s}}

\providecommand{\figures}{false}
{ \ifthenelse{\equal{\figures}{false}} {#1}{\[ {\rm Figure \ missing !} \]} }{}
\def\id{\mathrm{id}}

\usepackage{mathtools}

\def\CA{\mathcal{A}}

\def\CQ{\mathcal{Q}}

\def\CT{\mathcal{T}}

\tikzstyle{tinydots}=[dash pattern=on \pgflinewidth off \pgflinewidth]
\tikzstyle{superdense}=[dash pattern=on 4pt off 1pt]

\newcommand{\mcM}{\mathcal{M}}

\newcommand{\mcI}{\mathcal{I}}

\newcommand{\mcD}{\mathcal{D}}

\newcommand{\mcQ}{\mathcal{Q}}

\newcommand{\mcJ}{\mathcal{J}}

\newcommand{\beq}{\begin{equation}}
\newcommand{\eeq}{\end{equation}}

\usepackage{empheq}

\newcommand{\mbn}{\mathbf{n}}

\def\Labe{\mathfrak{e}}

\def\Labn{\mathfrak{n}}

\def\${|\!|\!|}

\newcounter{theorem}

\newtheorem{ex}[theorem]{Example}

\newenvironment{DIFnomarkup}{}{} 

\newtheorem{assumption}{Assumption}
 
\theorembodyfont{\rmfamily}

\newfont{\indic}{bbmss12}

\def\PPi{\boldsymbol{\Pi}}

\def\Nabla_#1{\nabla_{\!#1}}

\makeatletter
\pgfdeclareshape{crosscircle}
{
  \inheritsavedanchors[from=circle] 
  \inheritanchorborder[from=circle]
  \inheritanchor[from=circle]{north}
  \inheritanchor[from=circle]{north west}
  \inheritanchor[from=circle]{north east}
  \inheritanchor[from=circle]{center}
  \inheritanchor[from=circle]{west}
  \inheritanchor[from=circle]{east}
  \inheritanchor[from=circle]{mid}
  \inheritanchor[from=circle]{mid west}
  \inheritanchor[from=circle]{mid east}
  \inheritanchor[from=circle]{base}
  \inheritanchor[from=circle]{base west}
  \inheritanchor[from=circle]{base east}
  \inheritanchor[from=circle]{south}
  \inheritanchor[from=circle]{south west}
  \inheritanchor[from=circle]{south east}
  \inheritbackgroundpath[from=circle]
  \foregroundpath{
    \centerpoint%
    \pgf@xc=\pgf@x%
    \pgf@yc=\pgf@y%
    \pgfutil@tempdima=\radius%
    \pgfmathsetlength{\pgf@xb}{\pgfkeysvalueof{/pgf/outer xsep}}%
    \pgfmathsetlength{\pgf@yb}{\pgfkeysvalueof{/pgf/outer ysep}}%
    \ifdim\pgf@xb<\pgf@yb%
      \advance\pgfutil@tempdima by-\pgf@yb%
    \else%
      \advance\pgfutil@tempdima by-\pgf@xb%
    \fi%
    \pgfpathmoveto{\pgfpointadd{\pgfqpoint{\pgf@xc}{\pgf@yc}}{\pgfqpoint{-0.707107\pgfutil@tempdima}{-0.707107\pgfutil@tempdima}}}
    \pgfpathlineto{\pgfpointadd{\pgfqpoint{\pgf@xc}{\pgf@yc}}{\pgfqpoint{0.707107\pgfutil@tempdima}{0.707107\pgfutil@tempdima}}}
    \pgfpathmoveto{\pgfpointadd{\pgfqpoint{\pgf@xc}{\pgf@yc}}{\pgfqpoint{-0.707107\pgfutil@tempdima}{0.707107\pgfutil@tempdima}}}
    \pgfpathlineto{\pgfpointadd{\pgfqpoint{\pgf@xc}{\pgf@yc}}{\pgfqpoint{0.707107\pgfutil@tempdima}{-0.707107\pgfutil@tempdima}}}
  }
}
\makeatother

\def\symbol#1{\textcolor{symbols}{#1}}

\def\decorate#1#2{
        \ifnum#2>0
    		\foreach \count in {1,...,#2}{
	       	let
				\p1 = (sourcenode.center),
                \p2 = (sourcenode.east),
				\n1 = {\x2-\x1},
				\n2 = {1mm},
				\n3 = {(1.3+0.6*(\count-1))*\n1},
				\n4 = {0.7*\n1}
			in 
        		node[rectangle,fill=symbols,rotate=30,inner sep=0pt,minimum width=0.2*\n2,minimum height=\n2] at ($(sourcenode.center) + (\n3,\n4)$) {}
				}
		\fi
        \ifnum#1>0
    		\foreach \count in {1,...,#1}{
	       	let
				\p1 = (sourcenode.center),
                \p2 = (sourcenode.east),
				\n1 = {\x2-\x1},
				\n2 = {1mm},
				\n3 = {(1.3+0.6*(\count-1))*\n1},
				\n4 = {0.7*\n1}
			in 
        		node[rectangle,fill=symbols,rotate=-30,inner sep=0pt,minimum width=0.2*\n2,minimum height=\n2] at ($(sourcenode.center) + (-\n3,\n4)$) {}
				}
		\fi
}

\tikzset{
    dectriangle/.style 2 args={
        triangle,
        alias=sourcenode,
        append after command={\decorate{#1}{#2}}
    },
    dectriangle/.default={0}{0},
}

\tikzset{
	cross/.style={path picture={ 
  		\draw[symbols]
			(path picture bounding box.south east) -- (path picture bounding box.north west) (path picture bounding box.south west) -- (path picture bounding box.north east);
		}},
root/.style={circle,fill=green!50!black,inner sep=0pt, minimum size=1.2mm},
        dot/.style={circle,fill=pageforeground,inner sep=0pt, minimum size=1mm},
        dotred/.style={circle,fill=pageforeground!50!pagebackground,inner sep=0pt, minimum size=2mm},
        var/.style={circle,fill=pageforeground!10!pagebackground,draw=pageforeground,inner sep=0pt, minimum size=3mm},
        kernel/.style={semithick,shorten >=2pt,shorten <=2pt},
        kernels/.style={snake=zigzag,shorten >=2pt,shorten <=2pt,segment amplitude=1pt,segment length=4pt,line before snake=2pt,line after snake=5pt,},
        rho/.style={densely dashed,semithick,shorten >=2pt,shorten <=2pt},
           testfcn/.style={dotted,semithick,shorten >=2pt,shorten <=2pt},
        renorm/.style={shape=circle,fill=pagebackground,inner sep=1pt},
        labl/.style={shape=rectangle,fill=pagebackground,inner sep=1pt},
        xic/.style={very thin,circle,draw=symbols,fill=symbols,inner sep=0pt,minimum size=1.2mm},
        g/.style={very thin,rectangle,draw=symbols,fill=symbols!10!pagebackground,inner sep=0pt,minimum width=2.5mm,minimum height=1.2mm},
        xi/.style={very thin,circle,draw=symbols,fill=symbols!10!pagebackground,inner sep=0pt,minimum size=1.2mm},
	xies/.style={very thin,rectangle,fill=green!50!black!25,draw=symbols,inner sep=0pt,minimum size=1.1mm},
	xiesf/.style={very thin,rectangle,fill=green!50!black,draw=symbols,inner sep=0pt,minimum size=1.1mm},
        xix/.style={very thin,crosscircle,fill=symbols!10!pagebackground,draw=symbols,inner sep=0pt,minimum size=1.2mm},
        X/.style={very thin,cross,rectangle,fill=pagebackground,draw=symbols,inner sep=0pt,minimum size=1.2mm},
	xib/.style={thin,circle,fill=symbols!10!pagebackground,draw=symbols,inner sep=0pt,minimum size=1.6mm},
	xie/.style={thin,circle,fill=green!50!black,draw=symbols,inner sep=0pt,minimum size=1.6mm},
	xid/.style={thin,circle,fill=symbols,draw=symbols,inner sep=0pt,minimum size=1.6mm},
	xibx/.style={thin,crosscircle,fill=symbols!10!pagebackground,draw=symbols,inner sep=0pt,minimum size=1.6mm},
	kernels2/.style={very thick,draw=connection,segment length=12pt},
	keps/.style={thin,draw=symbols,->},
	kepspr/.style={thick,draw=connection,->},
	krho/.style={thin,draw=symbols,superdense,->},
	krhopr/.style={thick,draw=connection,superdense},
	triangle/.style = { regular polygon, regular polygon sides=3},
	not/.style={thin,circle,draw=connection,fill=connection,inner sep=0pt,minimum size=0.5mm},
	diff/.style = {very thin,draw=symbols,triangle,fill=red!50!black,inner sep=0pt,minimum size=1.6mm},
	diff1/.style = {very thin,dectriangle={1}{0},fill=red!50!black,draw=symbols,inner sep=0pt,minimum size=1.6mm},
	diff2/.style = {very thin,dectriangle={1}{1},fill=red!50!black,draw=symbols,inner sep=0pt,minimum size=1.6mm},
		diffmini/.style = {very thin,rectangle,fill=black,draw=black,inner sep=0pt,minimum size=0.75mm},
	 kernelsmod/.style={very thick,draw=connection,segment length=12pt},
	 rec/.style = {very thin,rectangle,fill=black,draw=black,inner sep=0pt,minimum size=2mm},
	cerc/.style={very thin,circle,draw=black,fill=symbols,inner sep=0pt,minimum size=2mm},
	stars/.style={very thin,star,star points=6,star point ratio=0.5, draw=black,fill=red,inner sep=0pt,minimum size=0.7mm},
	>=stealth,
        }
        \tikzset{
root/.style={circle,fill=black!50,inner sep=0pt, minimum size=3mm},
        circ/.style={circle,fill=white,draw=black,very thin,inner sep=.5pt, minimum size=1.2mm},
        round1/.style={fill=white,outer sep = 0,inner sep=2pt,rounded corners=1mm,draw,text=black,thin,minimum size=1.2mm},
          circ1/.style={circle,fill=red!10,draw=red,very thin,inner sep=.5pt, minimum size=1.2mm},
        rect/.style={fill=white,outer sep = 0,inner sep=2pt,rectangle,draw,text=black,thin,minimum size=1.2mm},
        rect1/.style={fill=white,outer sep = 0,inner sep=2pt,rectangle,draw,text=black,thin,minimum size=1.2mm},
        round2/.style={fill=red!10,outer sep = 0,inner sep=2pt,rounded corners=1mm,draw,text=black,thin,minimum size=1.2mm},
       round3/.style={fill=blue!10,outer sep = 0,inner sep=2pt,rounded corners=1mm,draw,text=black,thin,minimum size=1.2mm}, 
        rect2/.style={fill=black!10,outer sep = 0,inner sep=2pt,rectangle,draw,text=black,thin,minimum size=1.2mm},
        dot/.style={circle,fill=black,inner sep=0pt, minimum size=1.2mm},
        dotred/.style={circle,fill=black!50,inner sep=0pt, minimum size=2mm},
        var/.style={circle,fill=black!10,draw=black,inner sep=0pt, minimum size=3mm},
        kernel/.style={semithick,shorten >=2pt,shorten <=2pt},
         diag/.style={thin,shorten >=4pt,shorten <=4pt},
        kernel1/.style={thick},
        kernels/.style={snake=zigzag,shorten >=2pt,shorten <=2pt,segment amplitude=1pt,segment length=4pt,line before snake=2pt,line after snake=5pt,},
		kernels1/.style={snake=zigzag,segment amplitude=0.5pt,segment length=2pt},
		rho1/.style={densely dotted,semithick},
        rho/.style={densely dashed,semithick,shorten >=2pt,shorten <=2pt},
           testfcn/.style={dotted,semithick,shorten >=2pt,shorten <=2pt},
           visible/.style={draw, circle, fill, inner sep=0.25ex},
        renorm/.style={shape=circle,fill=white,inner sep=1pt},
        labl/.style={shape=rectangle,fill=white,inner sep=1pt},
        xic/.style={very thin,circle,fill=symbols,draw=black,inner sep=0pt,minimum size=1.2mm},
        xi/.style={very thin,circle,fill=blue!10,draw=black,inner sep=0pt,minimum size=1.2mm},
	xib/.style={very thin,circle,fill=blue!10,draw=black,inner sep=0pt,minimum size=1.6mm},
	xie/.style={very thin,circle,fill=green!50!black,draw=black,inner sep=0pt,minimum size=1mm},
	xid/.style={very thin,circle,fill=symbols,draw=black,inner sep=0pt,minimum size=1.6mm},
	edgetype/.style={very thin,circle,draw=black,inner sep=0pt,minimum size=5mm},
	nodetype/.style={very thick,circle,draw=black,inner sep=0pt,minimum size=5mm},
	kernels2/.style={very thick,draw=connection,segment length=12pt},
clean/.style={thin,circle,fill=black,inner sep=0pt,minimum size=1mm},	not/.style={thin,circle,fill=symbols,draw=connection,fill=connection,inner sep=0pt,minimum size=0.8mm},
	>=stealth,
        }

\makeatletter
\def\DeclareSymbol#1#2#3{%
	\expandafter\gdef\csname MH@symb@#1\endcsname{\tikzsetnextfilename{symbol#1}%
	\tikz[baseline=#2,scale=0.15,draw=symbols,line join=round]{#3}}%
	\expandafter\gdef\csname MH@symb@#1s\endcsname{\scalebox{0.75}{\tikzsetnextfilename{symbol#1}%
	\tikz[baseline=#2,scale=0.15,draw=symbols,line join=round]{#3}}}%
	\expandafter\gdef\csname MH@symb@#1ss\endcsname{\scalebox{0.65}{\tikzsetnextfilename{symbol#1}%
	\tikz[baseline=#2,scale=0.15,draw=symbols,line join=round]{#3}}}%
	}
\def\<#1>{\ifthenelse{\boolean{mmode}}{\mathchoice{\csname MH@symb@#1\endcsname}{\csname MH@symb@#1\endcsname}{\csname MH@symb@#1s\endcsname}{\csname MH@symb@#1ss\endcsname}}{\csname MH@symb@#1\endcsname}}
\makeatother

 \def\1{\mathbf{\symbol{1}}}

\def\one{\mathbf{1}}

\DeclareMathAlphabet{\mathpzc}{OT1}{pzc}{m}{it}

\let\d\partial

\def\eqref#1{(\ref{#1})}

\def\dint{\mathrm{d}}

\makeatletter 
\newcommand*{\bigcdot}{}
\DeclareRobustCommand*{\bigcdot}{%
  \mathbin{\mathpalette\bigcdot@{}}%
}
\newcommand*{\bigcdot@scalefactor}{.5}
\newcommand*{\bigcdot@widthfactor}{1.15}
\newcommand*{\bigcdot@}[2]{%
  \sbox0{$#1\vcenter{}$}
  \sbox2{$#1\cdot\m@th$}%
  \hbox to \bigcdot@widthfactor\wd2{%
    \hfil
    \raise\ht0\hbox{%
      \scalebox{\bigcdot@scalefactor}{%
        \lower\ht0\hbox{$#1\bullet\m@th$}%
      }%
    }%
    \hfil
  }%
}
\makeatother

\def\act{\bigcdot}

\tcbset
{colframe=boxcolor,colback=symbols!7!pagebackground,coltext=pageforeground,
fonttitle=\bfseries,nobeforeafter,center title,size=fbox,boxsep=1.5pt,
top=0mm,bottom=0mm,boxsep=0mm,tcbox raise base}

\def\two{{\<generic>\kern0.05em\<genericb>}}
\def\twoI{{\<Ito>\kern0.05em\<Itob>}}

\def\mail#1{\burlalt{#1}{mailto:#1}}

\usepackage{thmtools}

\begin{document}

\title{Recentering with Malliavin derivative}

\author{Yvain Bruned and Aurélien Minguella}
\institute{ 
 Universite de Lorraine, CNRS, IECL, F-54000 Nancy, France
  \\
Email:\ \begin{minipage}[t]{\linewidth}
\mail{yvain.bruned@univ-lorraine.fr}
\\  \mail{aurelien.minguella@univ-lorraine.fr}
\end{minipage}}

\maketitle 

\begin{abstract}
We provide an algebraic unification of the spectral gap proofs of the convergence of the renormalised model for regularity structures. We show that the key recentering map $\dint \Gamma$ used in the literature for adjusting the recentering of the model is given via equivalent characterisations. 
\end{abstract}
\setcounter{tocdepth}{2}
\setcounter{secnumdepth}{4}
\tableofcontents

 \section{Introduction}
 
 One of the main challenges in showing the local well-posedness of a large class of singular stochastic partial differential equations (SPDEs) was to describe systematically an ansatz of the solution formed by iterated stochastic integrals that converge after a suitable renormalisation procedure. This has been addressed in great generality within the theory of regularity structures introduced by Martin Hairer in \cite{reg}. One uses local expansions where the iterated integrals of the ansatz are localised around a base point and play the role of monomials. The combinatorial structures that describe these integrals are based on Hopf algebras of decorated trees and forests given in \cite{reg,BHZ}, which allow one to build renormalised models, \ie, a collection of localised renormalised stochastic integrals. The renormalisation is implemented following the celebrated BPHZ scheme developed in QFT for the renormalisation of Feynman diagrams (see \cite{BP57,KH69,WZ69}).
 Then, in \cite{CH16}, the convergence is shown by computing the cumulants of these renormalised stochastic integrals, which produces Feynman diagrams. These diagrams also incorporate the recentering procedure which makes the proof of convergence quite tricky as one has to deal with two analytical procedures simultaneously: renormalisation and recentering. The work \cite{BCCH} shows that, given the renormalised ansatz, described via a B-series, one can obtain the renormalised equation and perform a fixed point argument to get local well-posedness. 
 
 In \cite{OSSW,LOT}, the authors propose a different combinatorial framework, multi-indices, to encode the regularity structures ansatz for scalar-valued equations. The main idea is to parametrise the solution ansatz via the non-linearity of the equation, which boils down to grouping stochastic integrals sharing the same coefficient in the local expansion. By doing so, one does not have direct access  to the stochastic integrals, and therefore to the Feynman diagrams, to show the convergence of the model. In \cite{LOTP}, the authors developed a recursive approach to prove the convergence of the model based on a spectral gap inequality. Let us briefly summarise their approach. One considers a semi-linear equation
 \begin{equs} \label{main_equation}
 	(\partial_t - \mathcal{L}) u = f(u, \nabla u ,...,  \xi)
 \end{equs}
 where $f$ is affine in the space-time noise $ \xi $ and depends on $u$ and its derivatives, $ \mathcal{L}$ is a differential operator satisfying a Schauder estimate, such as the Laplacian. One wants to control the $p$-th moment of the recentered distribution $ \Pi_x \tau $ indexed by a combinatorial object $\tau$, that could be a decorated tree or a multi-index. These terms appear in the local expansion of the right hand side of \eqref{main_equation}:
 \begin{equs}
 	f(u, \nabla u ,...,\xi) = \sum_{\tau \in \CT} \frac{\Upsilon_f[\tau](u,\nabla u,...)(x)}{S(\tau)} \Pi_x \tau   + R_{\CT,x}
 \end{equs}
 where $ \CT $ is a finite combinatorial set, $ R_{\CT,x} $ is a Taylor remainder,  $S(\tau)$ is a symmetry factor associated with $\tau$, and $\Upsilon_{f}[\tau]$ are elementary differentials.
 The main idea of \cite{LOTP} is to use the following spectral gap inequality,
 \begin{equs}
 	\mathbb{E}(|\Pi_x \tau|^p)^{\frac{1}{p}} \leq |\mathbb{E}(\Pi_x \tau)| + \mathbb{E}( \Vert \delta \Pi_x \tau \Vert_{\star}^p )^{\frac{1}{p}},
 \end{equs}
 where $ \Vert \cdot \Vert_{\star} $ is a suitable functional norm, and $\delta \Pi_x  $ is the Malliavin derivative of the model $\Pi_x$. Then, one needs to have the control of the first moment of $\Pi_x \tau$ which is given by the BPHZ renormalisation scheme. The construction of $ \delta \Pi_x \tau $  is performed via a reconstruction argument. Indeed, taking the Malliavin derivative boils down to taking an abstract derivative on the combinatorial objects, which can be read in the commutation relation
 \begin{equs} \label{commutation_delta}
 	\delta \Pi_x \tau = \Pi_x D_{\Xi} \tau,
 \end{equs}
 where $ D_{\Xi} $ changes a decoration associated with a noise $\xi$ into a decoration associated with $\delta \xi$, the noise with a Malliavin derivative. This identity has been proved for decorated trees renormalised models in \cite[Proposition 4.1]{BN23}. Then, taking into account that $ \delta \xi $ has better regularity than $\xi$, one can observe that $ D_{\Xi} $ is a linear combination of combinatorial objects with positive degree. This is the reason why we can apply a reconstruction argument for this term. The success of this argument depends on a recursive structure  where one has to construct the model for $\sigma$ of smaller size than $\tau$.
 
 Different proofs exploiting this idea exist in the literature. The paper \cite{BH25} gives an overview of the various proofs for the convergence of the models.  We briefly present the spectral gap proofs
 \begin{itemize}
 	\item The original proof in \cite{LOTP} was conducted for quasilinear equations with multi-indices. This proof was reused to construct the model in \cite{T25} and the thin-film equation in \cite{GT}. The $ \Phi^4$ model is treated in \cite{BOT25}.
 	\item The proof in \cite{HS23} uses decorated trees and the notion of pointed modelled distribution for the reconstruction which looks similar to the analytical part of \cite{LOTP}. Their result works for a broad class of singular SPDEs.
 	\item The work \cite{BH23} adds another parameter to the model to measure integrability, which allows one to give a different version of the analytical arguments.
 \end{itemize}
 Let us stress that around the time when \cite{LOT} was completed, Pawel Duch in \cite{Duc21} also proposed a recursive approach to the convergence of stochastic integrals inspired by the Polchinski flow \cite{P84}. 
 
 As noticed in \cite[Definition 4.3]{BN23}, one can identify a modelled distribution whose reconstruction will yield $\delta \Pi_x \tau$, given by 
 \begin{equs} \label{dGamma}
 	\dint\Gamma_{yx} \tau=\mcQ_0\left(\Gamma_{yx}\otimes f_x\right)\hat\Delta D_\Xi\tau,
 	\end{equs}
 	where $ \Gamma_{xy} $ allows one to change the base point of the expansion, \ie
 	\begin{equs}
 		\Pi_x \Gamma_{xy} = \Pi_y.
 	\end{equs}
 	 The map $f_x$ is a character that defines the model $\Pi_x$ via a coproduct introduced in \cite{reg}
 	 \begin{equs}
 	 	\Pi_x = \left( \PPi \otimes f_x \right) \Delta
 	 \end{equs}
 where $\PPi$ is called a pre-model. The term $\PPi \tau$ is the stochastic iterated integral associated with $\tau$ without recentering. The map $\CQ_0$ is a projection that keeps only combinatorial terms without the abstract Malliavin derivative. Finally, $ \hat{\Delta} $ is the map associated with $\Delta$ that is used to shorten the recentering. Indeed, one can define a model $ \Pi^0_x $ that will not distinguish between the noise $\xi$ and $\delta \xi$ for the recentering and one has from \cite[Proposition 3.7]{BN23}
 \begin{equs}
 	\Pi^0_x = \left( \Pi_x \otimes f_x \right) \hat{\Delta}
 \end{equs}
 which is exactly the identity used in \cite[Lemma 8]{BH23}. Here, the model $\Pi_x^0 $ does not take into account the Malliavin derivative in the recentering which is not the case for $ \Pi_x $. In the sequel, the model $\Pi_x$ is replaced by two models $ \Pi_x^1 $ and $\Pi_x^2$ depending on the degree associated with the noise with a Malliavin derivative. We also replace $ \hat{\Delta} $ by suitable maps $ \hat{\Delta}_1$ and $ \hat{\Delta}_2 $. In the end, when one reconstructs the modelled distribution $\dint\Gamma_{\cdot x} \tau$ in the case of a smooth model, one gets 
 \begin{equs}
 	(\mathcal{R} \dint\Gamma_{\cdot x} \tau)(y) = (\Pi_y \dint \Gamma_{yx} \tau )(y) = (\delta \Pi_x \tau)(y),
 \end{equs}
where we have used the formula $ \mathcal{R}(\cdot)(y) = (\Pi_y \cdot)(y) $ for the reconstruction operator in the smooth case. The previous equality was proved in \cite[Theorem 4.5]{BN23}.
 Despite the dictionary proposed in \cite{BN23} in order to rewrite the main algebraic arguments of \cite{LOTP} using decorated trees not on the dual side, it is not clear that the $\dint \Gamma$ defined in \cite{BN23} corresponds to the one in \cite{LOTP}, nor how it is related to the construction of \cite{HS23}. The main result of this paper is to provide a clarification and to show that variations of the same map are used in all the spectral gap proofs. 
\begin{theorem} \label{main_theorem}
	 The algebraic parts of the convergence proofs via the spectral gap inequality in  \cite{LOT,HS23,BH23} rely on a recentering map $\dint\Gamma$ whose expression on decorated trees is given by \eqref{dGamma}
	\end{theorem} 
	The proof of this theorem is decomposed into two parts:
	\begin{itemize}
		\item Theorem \ref{th:local without Rnew} shows  that one has the identity \cite[Theorem 4.5]{BN23} but at higher order 
		 \begin{equs}
		 	 ( \partial^n \Pi_y \dint \Gamma_{yx} \tau )(y) = ( \partial^n\delta \Pi_x \tau)(y),
		 	\end{equs}
		 	where $n \in \mathbb{N}^{d+1}$ and the identity is true for every $n$ such that $ |n|_{\s} $ the norm of the vector $n$ is smaller than a certain quantity that depends on the equation. 
		This is exactly the characterisation provided by \cite[Section 5.4]{LOTP} and \cite[Section 2.6]{BOT25}.
		\item Theorem \ref{Thm_Martin} shows that a variant of $\dint \Gamma$ satisfies the recursive definition of the main algebraic map $ \tilde{f}_x^{\tau} $ considered in \cite[Section 4.1]{HS23}.
		The proof  relies on a key identity \eqref{backbone} given in \cite[Rem. 14.25]{FrizHai}.
	 It is also related to the recursive formula of $ \Gamma_{xy} $ in terms of $(\Pi_x \cdot)(y)$ given in \cite[Prop. 3.13]{BR18}.
		\end{itemize}
 	Let us briefly summarise the content of this paper by presenting the various sections.
 	In Section \ref{Sec::2},  we recall the main definition (see Definition \ref{def_decorated}) for the decorated trees used for solving singular SPDEs. We introduce two types of noise in order to distinguish the noises when they carry a Malliavin derivative. We define different degrees in \eqref{degree_def} associated with these trees that allows to take  or to not take into account the Malliavin derivative in the recentering. After defining the abstract Malliavin derivative $D_{\Xi}$ in Definition \ref{malliavin_derivative}, we introduce Assumption \ref{assumpt1} satisfied by a large class of singular SPDEs (see Example \ref{ex_assumpt}). This assumption guarantees that a planted tree containing a Malliavin derivative  is of positive degree. This is crucial in the sequel for using the reconstruction argument and the characterisation of the map $\dint \Gamma$. Then, we recall the co-action $ \Delta_i $, the coproduct $ \Delta_i^+ $, the preparation $R$ (see Definition \ref{DefnPreparationMap} and Assumption \ref{assumpt3}), the pre-model $\PPi$ needed for defining a renormalised model $\Pi_x^i$. Here, the index $i$ underlines the fact that we are using the degree $  \deg_i $ for the model. We finish the section with the definition of the central map of the paper $ \dint \Gamma $ (see Definition \ref{maindef}) together with a Hopf algebraic interpretation of this map and we illustrate it on examples (see Example \ref{ex_2}).
 	
 	In Section \ref{Sec::3}, we start to show Theorem \ref{th:local without Rnew}, the characterisation of $\dint \Gamma$ via derivatives. Its proof follows the same steps  as in the proof of \cite[Theorem 4.5]{BN23} and relies crucially on Assumption \ref{malliavin_derivative}. We finish by proving in Proposition \ref{propeqdual} that we obtain the characterisation on the dual side which was introduced in  \cite[Section 5.4]{LOTP}. One can also see it in  the lecture notes \cite[Section 2.7]{BOT25}. The second part of the section is dedicated to the proof of Theorem \ref{Thm_Martin}, the recursive characterisation of $\dint \Gamma$. Its proof boils down to the use of commutative identity between $ \Gamma_{yx} $ and $ \CI_{a} $ (see \eqref{backbone}).

 \subsection*{Acknowledgements}
 
 {\small
 	Y.B. and A.M. gratefully acknowledge funding support from the European Research Council (ERC) through the ERC Starting Grant Low Regularity Dynamics via Decorated Trees (LoRDeT), grant agreement No.\ 101075208. Views and opinions expressed are however those of the author(s) only and do not necessarily reflect those of the European Union or the European Research Council. Neither the European Union nor the granting authority can be held responsible for them.  The motivation for this work emerged while Y.B. was in residence at the Simons Laufer Mathematical Sciences Institute in Berkeley, California, during the Fall 2025 semester. This residence was partially supported by the National Science Foundation under Grant No. DMS-2424139. Y. B. thanks Lorenzo Zambotti for fruitful discussions and for having organised a working group on the diagram-free approach to regularity structures.}

 \section{Decorated trees and recentering map}
 \label{Sec::2}
 
We start the section by recalling the main definitions of \cite{BN23}.
We pick three symbols $\mcI$, $\Xi$, and $\dot\Xi$ and define with it, a set of edge decorations $ \mathcal{D} \coloneqq \lbrace \mcI,\,\Xi,\,\dot\Xi \rbrace \times \mathbb{N}^{d+1}$. The first of the symbols $\mcI$ represents convolution with a kernel that comes from the differential operator of the equation one is interested in, and the symbols $\Xi,\,\dot\Xi$ represent the noise term and an infinitesimal perturbation of it.

\begin{definition} \label{def_decorated}
A decorated tree over $\mathcal{D}$ is a $3$-tuple of the form  $\tau_{\Labe}^{\Labn} =  (\tau,\Labn,\Labe)$ where $\tau$ is a non-planar rooted tree with node set $N_\tau$ and edge set $E_\tau$. The maps $\Labn : N_{\tau} \rightarrow \mathbb{N}^{d+1}$ and $\Labe : E_\tau{\tiny } \rightarrow \mathcal{D}$ are node and edge decorations, respectively.
\end{definition}

We use $T$ to denote the set of decorated trees and $ \mathcal{T} $ is its linear span.  We define a binary tree product by 
\begin{equation}  \label{treeproduct}
 	(\tau,\Labn,\Labe) \cdot  (\tau',\Labn',\Labe') 
 	= (\tau \cdot \tau',\Labn + \Labn', \Labe + \Labe')\;, 
\end{equation} 
where $\tau \cdot \tau'$ is the rooted tree obtained by identifying the roots of $ \tau$ and $\tau'$. The sums $ \Labn + \Labn'$ mean that decorations are added at the root and extended to the disjoint union by setting them to vanish on the other tree. Each edge and vertex of both trees keeps its decoration, except the roots which merge into a new root decorated by the sum of the previous two decorations.

\begin{enumerate}
   \item[(i)] An edge decorated by  $ (\mcI,a) \in \mathcal{D} $  is denoted by $ \mcI_{a} $. The symbol $  \mcI_{a} $ is also viewed as the operation that grafts a tree onto a new root via a new edge with edge decoration $ a $. The new root at hand remains decorated with $0$. 
   
   \item[(ii)]  An edge decorated by $(\Xi,0)$ or $(\dot\Xi,0)$ is denoted by $\Xi$ or $\dot\Xi$ respectively.

   \item[(iii)] A factor $ X^k$   encodes a single node  $ \bullet^{k} $ decorated by $ k \in \mathbb{N}^{d+1}$. We write $ X_i$, $ i \in \lbrace 0,1,\ldots,d\rbrace $, to denote $ X^{e_i}$. Here, we have denoted by $ e_0,...,e_d $ the canonical basis of $ \mathbb{N}^{d+1} $. The element $ X^0 $ which encodes $\bullet^0$, will be denoted by $\one$. The space of all the monomials $X^{k}$ will be denoted by $\bar{T}$, and its linear span by $\bar{\CT}$.
 \end{enumerate}
 
 We assume that the noise type edges appear always with the decorations $ \Xi $ or $\dot{\Xi}$ and always connected to a leaf.
  We denote $T_0$ the set of decorated trees containing no symbol $\dot\Xi$, and $T_1$ the set of decorated trees containing at least one symbol $\dot\Xi$. We consequently define their spans $\CT_0$ and $\CT_1$ respectively.
 
 Using this symbolic notation any decorated tree $ \tau \in T $ can be represented as:
 \begin{equs}\label{eq:treeform}
 \tau = X^k\Xi^l \dot\Xi^m  \prod_{i=1}^n \mcI_{a_i}(\tau_i),
 \end{equs}
 where $ \prod_i $ is the tree product, $ k \in \mathbb{N}^{d+1} $, $l$ , $m \in \mathbb{N} $. In relevant applications, and particularly for the one we have in mind for this article, a product of noises is not allowed and one can only consider the cases for which $ l + m \le 1$. Also, we assume that $ \CI_{a}(X^k)  =0 $ for every $k \in \mathbb{N}^{d+1}$. This comes from the fact that on the analytical side one considers a kernel which will cancel out monomials up to a sufficiently large degree depending on the applications.  A planted tree is a tree of the form $ \mcI_a(\tau) $ meaning that there is only one edge connecting the root to the rest of the tree. We also introduce abstract derivatives $\CD_p$, for $n\in\mathbb{N}^{d+1}$ by requiring that
\begin{equs}
\CD_n\CI_a(\tau)=\CI_{a+n}(\tau),\quad \CD_nX^k = \begin{cases}
\frac{k!}{(k-n)!}X^{k-n} & k \ge n \\
0 & \mbox{otherwise},
\end{cases}
\end{equs}
and finally extending to all of $T$, by Leibniz rule.
Upon the noises we define the degree maps $\deg_0$, $\deg_1$ and  $\deg_2$ by setting
\begin{equs}
\deg_j(\Xi)=\alpha,~\deg_0(\dot\Xi)=\alpha,~\deg_1(\dot\Xi)=\alpha+\frac{d+2}{2},~\deg_2(\dot\Xi)=\alpha+\frac{d+2}{2}+ 2\kappa,
\end{equs}
where $ \alpha $ is the space-time regularity of the noise $ \xi $ encoded by $ \Xi $ and the presence of $\tfrac{d+2}{2}(+ 2\kappa)$ in $\deg_1(\dot\Xi)$ and $\deg_2(\dot\Xi)$ is to keep track of the gain in regularity that comes from taking a Malliavin derivative. $\kappa$ is a small positive constant. From there the degree maps are extended to all of the $T$ by decreeing that for $j\in\{0,1,2\}$:
\begin{equs}
	\label{degree_def}
	\begin{aligned}
&\deg_j(X_0)= 2,&&\deg_j(X_i) = 1,\,\,\,\text{for }i\neq 0, \\
&\deg_j(\sigma\tau)=\deg_j(\sigma)+\deg_j(\tau),\qquad&&\deg_j\left(\mcI_a(\tau) \right)= \deg_j(\tau) + 2 - |a|_{\s},
\end{aligned}
\end{equs}
where we have used parabolic scaling $\s =  (2,1,\hdots,1)$, and some multi-index $n\in\mathbb{N}^{d+1}$ to define the quantity
 \begin{equs}
|n|_{\s}\coloneqq \sum_{i=0}^d \s_i n_i.
\end{equs}
Notice that this means for a multi-index $n$, $\deg_j(X^n) = |n|_{\s}$. We also denote $|\cdot|_\Xi$ the number of noises in $\tau$, regardless of their type. Moreover, the indices $0$ and $1$ have been swapped compared to their counterparts in \cite{BN23}. This is because we need a third degree and we decided to sort them by increasing value.

 Given a subset $ E \subset T $, we define the following set
 \begin{equs}
 	E^{+,i} := \left\lbrace X^k \prod_{j} \mathcal{I}^{+,i}_{a_j}(\tau_j) \, \middle| \, \deg_i(  \mathcal{I}_{a_j}(\tau_j) ) > 0, \, \tau_j \in E, \, k \in \mathbb{N}^{d+1}   \right\rbrace. 
 \end{equs}
The new symbol $ \mathcal{I}^{+,i}_{a_j} $ has been chosen in order to stress the difference between the set above and $ E $, which occurs especially when $ E = T_0 $. There is no constraint on the product at the root for the decorated trees in $ T_0^{+,i} $. The symbol can be extended to any element of $ T $ by sending to zero the trees of negative degree:
\begin{equs}
	\CI^{+,i}_a(\tau) = 0, \quad \deg_i(  \mathcal{I}_{a}(\tau) ) \le 0.
\end{equs}
As above, we denote by $ \mathcal{T}^{+,i} $ the linear span  of $ T^{+,i} $.

We define classical combinatorial coefficients on $T_1$, the symmetry factors. They given inductively on a tree of the form $\tau=X^k\Xi^l\dot\Xi^m\prod_{i=1}^n\CI_{a_i}(\tau_i)^{\beta_i}$, where the couples $(a_i,\tau_i)$ are pairwise disjoint, by
\begin{equation}
S(\tau)=k!\prod_{i=1}^n\beta_i!S(\tau_i)^{\beta_i},
\end{equation}
and $S(\one)=S(\Xi)=S(\dot\Xi)=1$. With this, we can also define a scalar product on $T_1\times T_1$ by
\begin{equation}
\langle\tau,\sigma\rangle=\one_{\tau=\sigma}S(\tau),
\end{equation}
and extend it by bilinearity on $\CT_1\times\CT_1$.

We introduce the abstract derivative on trees that stands for the algebraic counterpart of the Malliavin derivative with respect to the noise. This is Definition 2.2 in \cite{BN23}.
\begin{definition} \label{malliavin_derivative}
The linear map $ D_{\Xi} \,: \, \mathcal{T}_0 \rightarrow \mathcal{T}_1 $ is defined as the derivation that turns $ \Xi $ of a given decorated tree into $ \dot{\Xi} $. It can be defined inductively on the tree structure by requiring
\begin{equs}
D_{\Xi} \left( X^k \Xi \prod_{i=1}^n \mathcal{I}_{a_i}(\tau_i) \right) = X^k \dot\Xi \prod_{i=1}^n &\mathcal{I}_{a_i}(\tau_i) + X^k \Xi  \sum_{j=1}^n \mathcal{I}_{a_j}(D_{\Xi} \tau_j) \prod_{i \neq j} \mathcal{I}_{a_i}(\tau_i),
\end{equs}
and also $D_{\Xi}\left(\Xi\right) = \dot\Xi$, and $D_{\Xi}\left(X^\mbn\right) = D_{\Xi}\left({\mathbf{1}}\right)=0$.
\end{definition}
 
We make the following crucial assumption on the space $ \mathcal{T}_0 $:
\begin{assumption} \label{assumpt1}
For every $ \tau  \in \mathcal{T}_0$, we assume that:
\begin{equs} \label{decomposition_D_Xi}
D_{\Xi} \tau = \sum_{i} \mathcal{I}_{a_i}(D_\Xi \tau_i') \tau_i + \dot{\Xi} \tau', 
\end{equs}
where the $ \tau_i, \tau_i', \tau' $ belongs to $ \mathcal{T}_0 $ and $ \tau' $ does not have any noise at the root. We denote by $N$ the biggest integer such that $\deg_1(\dot\Xi)+2-N>0$. This latter quantity is related to the Malliavin derivative of the solution of the linear equation. We suppose that, for every $|n|_{\s}\leq N$,
\begin{equs} \label{degree_positive}
\deg_1( \mathcal{I}_{a_i+n}(D_\Xi \tau_i') ) > 0,
\end{equs}
where in the later identity, we have made an abuse of notation as $\deg_1$ is not defined for a linear combination of decorated trees. Here, we know that $D_\Xi \tau_i'$ is a linear combination of decorated trees of the same degree. Therefore,  we denote by  $\deg_1(D_\Xi \tau_i' )$ their degree.
\end{assumption}

This assumption is an extension of \cite[Assumption 3]{BN23} where \eqref{degree_positive} is only assumed for $n=0$. This new assumption is needed for showing Theorem \ref{th:local without Rnew} and it is satisfied by most natural models. We consider few models in the following example:

\begin{example} \label{ex_assumpt}
	Let us consiser the $\Phi_{4-\kappa}^{4}$ model with $\kappa > 0$. Its dynamics is given by
	\begin{equs}
		(\partial_t - \Delta) u = - u^3 + \xi, \quad (t,x) \in \mathbb{R}_+ \times \mathbb{T}^4
	\end{equs}
	where the space-time noise $\xi$ has its trajectories in $ \mathcal{C}^{-3 + \kappa} $. Then, one defines $ T_0 $ as
	\begin{equs}
		T_0 = \{\Xi, X^k \CI(\tau_1), X^k \CI(\tau_1) \CI(\tau_2), \CI(\tau_1) \CI(\tau_2) \CI(\tau_3),\, k \in \mathbb{N}^{d+1}, \, \tau_i \in T_0   \}.
		\end{equs}
	We set $ \alpha = -3 + \kappa $. Then, one realises that the decorated tree of lowest degree of the form $ \CI(\tau) $ with $\tau \in T_0$ is given by $ \CI(\Xi) $. One has
	\begin{equs}
		\deg_1( \mathcal{I}(D_\Xi \, \Xi) ) & = \deg_1( \mathcal{I}(\dot{\Xi}) ) = 2 + \alpha + \frac{2+ d}{2}  = 2 + \kappa,
		\\
		\deg_2( \mathcal{I}(D_\Xi \, \Xi) ) & = \deg_2( \mathcal{I}(\dot{\Xi}) )   = 2 + 3 \kappa
		\end{equs} 
		where we have used that $d=4$. Then, one has $N=2$. For the specific model $ \Phi_3^4 $, one considers space-time white noise which means trajectories in $ \alpha = -\frac{5}{2} - \kappa $, then one has
		\begin{equs}
			\deg_1( \mathcal{I}(D_\Xi \, \Xi) )  =  2 - \kappa
			\quad 
			\deg_2( \mathcal{I}(D_\Xi \, \Xi) )  = 2 + \kappa.
			\end{equs}
			In this case, one has $N=1$.
	The second model we consider is the generalised KPZ equation given by
	\begin{equs}
		(\partial_t - \partial_x^2) u = f(u) (\partial_x u)^2 + g(u) \xi, \quad (t,x) \in \mathbb{R}_+ \times \mathbb{T}
	\end{equs}
	where $\xi$ is the space-time white noise. Therefore, $ \alpha = - \frac{3}{2} - \kappa$.
	The set $T_0$ is given by 
	\begin{equs}
		T_0 & = \Biggl\{ X^k \tau,  \prod_{i=1}^n \CI(\tau_i) \Xi,  \prod_{i=1}^n \CI(\tau_i)  \prod_{i=1}^p\CI_{(0,1)}(\sigma_i), \\ & \, p \in \{  0,1,2,3 \}, n \in \mathbb{N},  k \in \mathbb{N}^{d+1}, \, \tau, \tau_i, \sigma_i \in T_0   \Biggl\}
	\end{equs}
	where by convention the previous product $\prod$ are equal to one when $ n=0 $ or $p=0$. Then, in this case the term with lowest degree of the form $\CI_{(0,1)}(\tau)$ is $\mathcal{I}_{(0,1)}(\Xi)$ which gives
	\begin{equs}
			\deg_1( \mathcal{I}_{(0,1)}(D_\Xi \, \Xi) )  =  2- \kappa
			\quad 
			\deg_2( \mathcal{I}(D_\Xi \, \Xi) )  = 2 + \kappa, \quad N = 1.
		\end{equs}
\end{example}

We recall the definition of the co-actions $ \Delta_i : \CT_i \rightarrow \CT_i \otimes \CT^{+,i} $ first defined in \cite{reg}.
\begin{equation} \label{coaction}
\begin{split}
&\Delta_i \mathbf{1} \coloneqq \mathbf{1} \otimes \mathbf{1},\qquad \Delta_i X_j = X_j \otimes \mathbf{1} + \mathbf{1}\otimes X_j, \\
& \Delta_j\Xi = \Xi\otimes\one,~~\Delta_0 \dot\Xi = \dot\Xi \otimes \mathbf{1},~~\Delta_1 \dot\Xi = \dot\Xi \otimes \mathbf{1},~~\Delta_2 \dot\Xi = \dot\Xi \otimes \mathbf{1}+\one\otimes\dot\Xi,\\
&\Delta_i\mcI_a(\tau) \coloneqq (\mcI_a\otimes \mathrm{id} )\Delta_i \tau + \hspace{-4mm}\sum_{\vert\ell +m\vert_{\s} < \deg_i(\mcI_a(\tau))} \frac{X^\ell}{\ell!}\otimes \frac{X^m}{m!}\mcI^{+,i}_{a+\ell+m}(\tau).
\end{split}
\end{equation}
We also supplement these coactions with coproducts $\Delta^{+}_{i}:\CT^{+,i}\rightarrow \CT^{+,i}\otimes \CT^{+,i}$.
\begin{equs} \label{def_delta_+} \begin{aligned}
\Delta^{+}_{i}\sigma\tau&=\Delta^{+}_{i}\sigma\Delta^{+}_{i}\tau, \\
\Delta^{+}_{i}\CI_a^{+,i}(\tau)&=\sum_{|\ell|_{\s}<\deg_i(\CI_a(\tau))}\left(\CI^{+,i}_{a+\ell}\otimes\frac{(-X)^{\ell}}{\ell!}\right)\Delta\tau+\one\otimes\CI^{+,i}_{a}(\tau).
\end{aligned}
\end{equs}
\color{black}

We recall the definition of the preparation maps, first introduced in \cite{BR18}, that we will use to define the renormalised models inductively.
\begin{definition}
 \label{DefnPreparationMap}
A preparation map is a linear map $
R : \CT \rightarrow \CT $ 
that fixes polynomials, noises, planted trees, and such that 
\begin{itemize}
   \item for each $ \tau \in \mathcal{T} $ there exist finitely many $\tau_j \in \mathcal{T}$ and constants $\lambda_j$ such that
\begin{equation} \label{EqAnalytical}
R \tau = \tau + \sum_j \lambda_j \tau_j, \quad\textrm{with}\quad \deg_0(\tau_j) \geq \deg_0(\tau) \quad\textrm{and}\quad |\tau_j|_{\Xi} < |\tau|_{\Xi}.
\end{equation} 
   \item one has 
 \begin{equation} \label{EqCommutationRDelta}
 ( R \otimes \mathrm{id}) \Delta_i = \Delta_i R.
 \end{equation}
 \item The map $ R $ commutes with $ D_{\Xi} $ in the sense that one has on $ \mathcal{T}_0 $:
 \begin{equs} \label{commute_D_Xi}
 R D_{\Xi} = D_{\Xi} R.
 \end{equs}
\end{itemize}
\end{definition}

The commutation property \eqref{commute_D_Xi} is where we differ from the original definition in the literature. It is meant to reflect the fact that by changing a noise $ \xi $ in some tree into an infinitesimal perturbation of itself $ \delta \xi $, it becomes less singular and therefore does not require any renormalisation. 
One natural choice for the preparation map $ R $ is given in \cite{BR18,BB21} by:
\begin{equs}\label{eq:prepmap}
R_{\ell}^{*} \tau =  \sum_{\sigma \in T^-_0} \frac{\ell(\sigma)}{S(\sigma)} \sigma \star \tau,
\end{equs}
where $ T_{0}^-  $ are elements in $ T_{0} $ with negative degree using $ \deg_1 $ or $\deg_2$, the associative product $ \star $ is related to the dual of $ \Delta_i $ and $ \ell : T_0^{-} \rightarrow \mathbb{R} $, which provides the renormalisation constants, is chosen in such a manner so that $ R_{\ell}$ has the various properties of a preparation map. The map $ R_{\ell} $ whose dual map is $ R_{\ell}^{*} $ has been first introduced in \cite{BR18} as an extraction-contraction procedure of a subtree with negative degree happening at the root. We make the following assumption on the preparation map.
\begin{assumption}\label{assumpt3}
We assume that the preparation map R is such that
\begin{equation}
R\CQ_0 = \CQ_0R,
\end{equation}
where $\CQ_0$ projects trees containing at least one Mallivin derivative $\dot\Xi$ to zero.
\end{assumption}
With a preparation map $R$ fixed, one is able to define the corresponding renormalised model $ \Pi^{i}_x $ for $ i \in \lbrace 0,1 \rbrace $ inductively as follows:
\begin{equs} \label{recursive_model}
	\begin{aligned}
\Pi_x^{i} \tau &  = \Pi^{i,\times}_x R \tau, \quad  \Pi^{i,\times}_x (\sigma \tau) = ( \Pi^{i,\times}_x \sigma)
 \times (\Pi^{i,\times}_x\tau), \nonumber \\ 
\big( \Pi_x^{i,\times}\mathcal{I}_a(\tau)\big)(y)& = \big(\d^{a} K  * { \Pi}_x^{i} \tau\big)(y)\nonumber  \\   &- \sum_{|k|_{\s} < \deg_i(\mathcal{I}_a(\tau)) }\frac{(y-x)^k}{k!}  \big(\d^{a+k}  K * { \Pi}_x^{i} \tau\big)(x),
\end{aligned}
\end{equs}
where $ K $ is now a compactly supported function obtained from the kernel associated with the singular SPDEs considered, with the seed given by:

\begin{equs}
\Pi_x^{i}\Xi = \xi,\quad&\Pi_x^{0}\dot\Xi=\delta\xi, \quad\Pi_x^{1}\dot\Xi=\delta\xi,\quad\Pi_{x}^{2}\dot\Xi=\delta\xi-\delta\xi(x), 
\\ \Pi_x^{i}X^{k} & = (\cdot - x)^{k},  \quad \Pi^{i,\times}_{x}\tau = \Pi_x^{i}\tau,\quad\text{for }\tau\in\{\Xi,\dot\Xi\}\cup\bar{T}.
\end{equs}

Before continuing, we need several more definitions. The first one is the following classical object in the regularity structures literature. It is set inductively by

\begin{equation}\label{eq:recf}\begin{split}
f_x^{i}(\mathbf{1})&=1,\qquad f_x^{i}(X_j) = -x_j\mbox{ for }j\in\{0,\hdots,d\},\\
&f_x^{i}(\CI_a^{+,i}(\tau))= - (\d^a K * \Pi_x^{i} \tau)(x).
\end{split}
\end{equation}
One can define the model $\Pi_x^{i}$ using the co-action $\Delta_i$ and $f_x^i$
\begin{equs}
	\label{algebraic_inter}
	\Pi_x^i = \left(\PPi \otimes f_x^i  \right) \Delta_i
\end{equs}
where the map $\PPi$ is the pre-model defined as
\begin{equs}
	\label{pre_model}
	\begin{aligned}
		~&\PPi \tau   = \PPi^{\times} R \tau, \quad  \PPi^{\times} (\sigma \tau) = ( \PPi^{\times} \sigma) \times (\PPi^{\times}\tau),  \\ 
		&\big( \PPi^{\times}\mathcal{I}_a(\tau)\big)(y) = \big(\d^{a} K  * { \PPi} \tau\big)(y) 
	\end{aligned}
\end{equs}
One can notice that this map does not depend on the degree map $  \deg_i$ as there is no recentering in its definition. Another map needed in the sequel is the re-expansion map $ \Gamma_{xy}^i $ given by
\begin{equs}
	\Gamma_{xy}^i  = \left( \id \otimes \gamma_{xy}^i \right) \Delta_i, \quad 
\end{equs}
where $ \gamma_{xy}^i $ is defined by
\begin{equs}
	\gamma_{xy}^i = \left(  f_x \CA_i^+ \otimes f_y \right) \Delta_i^{\!+},
	\end{equs}
	where $ \CA_i^+ : \CT^{+,i} \rightarrow \CT^{+,i} $ is the antipode for the Hopf algebra $\CT^{+,i}$ satisfying 
	\begin{equs}
		\mathcal{M}^{+,i} \left( \CA^{+}_i \otimes \id \right) \Delta_i^+ = \mathcal{M}^{+,i} \left( \id \otimes \CA^{+}_i \right) \Delta_i^+ = \one \one^{*},
	\end{equs}
	and $ 	\mathcal{M}^{+,i} : \CT^{+,i} \otimes \CT^{+,i}  \rightarrow \CT^{+,i} $ is the multiplication map in $ \CT^{+,i}$, $\one^{*} : \CT^{+,i} \rightarrow \R $ is the co-unit being non-zero and equal to $1$ on $ \one $.
	One has also the recursive definition for the antipode $\CA^+_i$ given in \cite[Lemma 8.39]{reg} by
	\begin{equs}
		\label{antipode_+_i}
		\begin{aligned}
	\CA_i^+ (\one) & = \one, \quad	\CA_i^+ (X_j) = - X_j, \quad \CA_i^+ (\sigma\tau  ) = \CA_i^+ (\sigma )  \CA_i^+ ( \tau ),
	\\
	\CA_i^+ (\CI^{+,i}_a(\tau)) & =
-  \sum_{\ell \in \N^{d+1}} \mathcal{M}^{+,i} \left( \CI^{+,i}_{a+\ell} \otimes \frac{X^{\ell}}{\ell!}  \CA^{+}_i \right) \Delta^{\!+}_i \tau.
\end{aligned}
	\end{equs}
We provide a commutation formula between  $ \Gamma_{xy}^i $ and $\CI_a$:
\begin{equation}\label{backbone}
	\Gamma^i_{xy}(\mcI_a+\mcJ_{a}^i(y))=(\mcI_a+\mcJ_{a}^i(x))\Gamma_{xy}^i
\end{equation}
where 
\begin{equation}
	\mcJ_a^i(x)\tau=\sum_{|k|_{\s}<\deg_i(\mcI_a\tau)}(\d^{a+k}K*\Pi_x^i\tau)(x)\frac{X^k}{k!}.
\end{equation}
The formula
\eqref{backbone} is given in \cite[Remark 14.25]{FrizHai}.
 It is also related to the recursive formula of $ \Gamma_{xy} $ in terms of $(\Pi_x \cdot)(y)$ given in \cite[Proposition 3.13]{BR18}.

Finally, we define two maps on $\CT_1$ by
\begin{equs}
	\label{EqDefnDeltaHatCirc}
\begin{aligned}
~&\hat{\Delta}_{1}\bullet = \bullet\otimes \mathbf{1}, \quad \textrm{ for }\bullet\in\big\{\mathbf{1}, X_j, \Xi,\dot\Xi\big\} , \\
&\hat\Delta_2\bullet=\bullet\otimes \mathbf{1}, \quad \textrm{ for }\bullet\in\big\{\mathbf{1}, X_j, \Xi\big\},\quad\hat\Delta_2\dot\Xi=\dot\Xi\otimes \mathbf{1}-\one\otimes\dot\Xi,\\
&\hat{\Delta}_{i} \mcI_a(\tau) = (\mcI_a\otimes\mathrm{id})\hat{\Delta}_{i} \tau \nonumber- \sum_{\deg_0(\mcI_a(\tau)) \leq \vert\ell\vert_{\s}} \frac{X^\ell}{\ell!}\otimes\mcM^{+,i}\big(\mcI^{+,i}_{a+\ell}\otimes\mathrm{id}\big)\hat{\Delta}_{i}\tau, 
\end{aligned}
\end{equs}
\noindent for $i\in\{1,2\}$ and $j\in\{0,\hdots,d\}$. This was firstly defined in  \cite[Proposition 3.8]{BN23}.

With these notations at hand, we are ready to define the central objects of this paper. The first one is \cite[Definition 4.3]{BN23}.

\begin{definition}\label{maindef}
For $\tau\in\CT_0$, we let
\begin{equation}
\dint\Gamma_{yx}^1\tau=\mcQ_0\left(\Gamma_{yx}^1\otimes f_x^1\right)\hat\Delta_1 D_\Xi\tau,
\end{equation}
where $\mcQ_0$ projects to zero all the trees containing the noise $\dot\Xi$,
\begin{equation}
\dint\Gamma_{yx}^2\tau=\left(\Gamma_{yx}^2\otimes f_x^2\right)\hat\Delta_2 D_\Xi\tau.
\end{equation}
\end{definition}

One can propose a different interpretation of the map $ \dint \Gamma^i $ similar to the Hopf algebraic formulation given in  \eqref{algebraic_inter}.
From \cite[Proposition 3.9]{BN23}, one has 
\begin{equs} \label{recrusive_gamma_i}
	\hat{\Delta}_i = \left( \id \otimes \hat{\Gamma}_i \right) \Delta_i
\end{equs}
where
\begin{equs} \label{inter_gamma}
	\begin{aligned}
		~&\hat{\Gamma}_{i} X_j = 0, \quad \hat{\Gamma}_i \one = \one, \quad   \hat{\Gamma}_{2}     \dot{\Xi}  = - \dot{\Xi}, \\ &	 \hat{\Gamma}_i \left( \sum_{\ell \in \mathbb{N}^{d+1}} \frac{X^\ell}{\ell!} \CI_{a+\ell}^{+,i}(\tau) \right)
	 = - \one_{\{ \deg_0(\CI_a (\tau)) \leq 0  \}}  \mathcal{M}^{+,i} \left( \CI_{a}^{+,i}\otimes \id \right) \hat{\Delta}_i \tau.
	 \end{aligned}
	\end{equs}
	Then, combining the two identities \eqref{recrusive_gamma_i} and \eqref{inter_gamma}, one gets
\begin{equs} \label{gamma_i}
		 \hat{\Gamma}_i \left( \sum_{\ell \in \N^{d+1}} \frac{X^\ell}{\ell!} \CI_{a+\ell}^{+,i}(\tau) \right)
	= - \one_{\{ \deg_0(\CI_a (\tau)) \leq 0  \}}  \mathcal{M}^{+,i} \left( \CI_{a}^{+,i}\otimes \hat{\Gamma}_i \right) \Delta_i \tau
\end{equs}
where in the previous equation we have made a change of basis in the definition of $ \hat{\Gamma}_i $ by considering $ \tilde{\CI}^{+,i}_a(\tau) : = \sum_{\ell \in \N^{d+1}} \frac{X^{\ell}}{\ell!} \CI_{a+\ell}^{+,i}(\tau) $
which is very close to the recursive formula for the antipode $  \CA_i^+$ \eqref{antipode_+_i}. Therefore, one has
\begin{equation}
	\dint\Gamma_{yx}^1\tau=\mcQ_0\left(\Gamma_{yx}^1\otimes f_x^1 \hat{\Gamma}_1 \right) \Delta_1 D_\Xi\tau,
\quad
	\dint\Gamma_{yx}^2\tau=\left(\Gamma_{yx}^2\otimes f_x^2 \hat{\Gamma}_2 \right)\Delta_2 D_\Xi\tau.
\end{equation}

\begin{example} \label{ex_2}
We give an example of these two expressions on a simple decorated tree given by $ \CI(\Xi) $ in the case of the generalised KPZ equation.
One has
\begin{equs}
	\hat{\Delta}_{i} \mcI(\dot{\Xi}) = (\mcI\otimes\mathrm{id})\hat{\Delta}_{i} \dot{\Xi} \nonumber- \sum_{\deg_0(\mcI(\tau)) \leq \vert\ell\vert_{\s}} \frac{X^\ell}{\ell!}\otimes\mcM^{+,i}\big(\mcI^{+,i}_{\ell}\otimes\mathrm{id}\big)\hat{\Delta}_{i} \dot{\Xi}.
\end{equs}
Then, one uses the fact that 
\begin{equs}
	(\mcI\otimes\mathrm{id})\hat{\Delta}_{i} \dot{\Xi}  = \CI(\dot{\Xi}) \otimes \one, \quad (\mcI^{+,i}\otimes\mathrm{id})\hat{\Delta}_{i} \dot{\Xi} = \CI^{+,i}(\dot{\Xi}) \otimes \one,
	\end{equs}
	and
	\begin{equs}
\deg_0(  \CI(\dot{\Xi}) )=  \frac{1}{2} - \kappa, \quad 		\deg_1(  \CI(\dot{\Xi})) = 2 - \kappa, \quad 	\deg_2(  \CI(\dot{\Xi})) = 2 +  \kappa
	\end{equs}
	to get
	\begin{equs}
			\hat{\Delta}_{1} \CI(\dot{\Xi}) & = \CI(\dot{\Xi}) \otimes \one - X_1 \otimes \CI_{(0,1)}(\dot{\Xi})
			\\ 
			\hat{\Delta}_{2} \CI(\dot{\Xi}) 
		&	= \CI(\dot{\Xi}) \otimes \one - X_1 \otimes \CI_{(0,1)}(\dot{\Xi})
		- \frac{X_1^2}{2} \otimes \CI_{(0,1)}(\dot{\Xi})
		- X_0 \otimes \CI_{(1,0)}(\dot{\Xi}).
		\end{equs}
		Then, we compute $ \Gamma_{yx}^i \CI(\dot{\Xi}) $. One has 
		\begin{equs}
			\CI(\Gamma_{yx}^i \dot{\Xi}) = \CI(\dot{\Xi}).
		\end{equs}
		From \eqref{backbone}, one gets
		\begin{equs}
			\Gamma^i_{yx}(\mcI(\dot{\Xi})) & = - \mcJ^i(x) \dot{\Xi}+ (\mcI+\mcJ^i(y))\Gamma_{yx}^i \dot{\Xi}
			\\ & =  \sum_{|k|_{\s}<\deg_i(\mcI(\tau))}
			\left(  (\d^{k}K*\delta \xi)(y)
			- (\d^{k}K*\delta \xi)(x) \right)\frac{X^k}{k!} + \CI(\dot{\Xi}).
		\end{equs}
		In the end, one has
		\begin{equs}
			\dint\Gamma^1_{yx}\CI(\dot\Xi)=(\d K*\delta\xi)(y)X+\left((K*\delta\xi)(y)+(y-x)(\d K*\delta\xi)(y)\right)\one.
		\end{equs}
		and 
				\begin{equs}
		&	\dint\Gamma^2_{yx}\CI(\dot\Xi) =
			(\d^2 K*\delta\xi)(y)X_1^2 + (\partial_t K*\delta\xi)(y) X_0
			\\ & +( (\d K*\delta\xi)(y)+ (y_1 -x_1)  (\d^2 K*\delta\xi)(y) ) X_1+  
			 ((K*\delta\xi)(y) \\ & +(y_1-x_1)(\d K*\delta\xi)(y) +
			\frac{(y_1-x_1)^2}{2}(\d^2 K*\delta\xi)(y) + (y_0 - x_0)(\partial_t K*\delta\xi)(y) )\one.
		\end{equs}
\end{example}
 
\section{Characterisations of the recentering map}\label{Sec::3}

 	\subsection{Derivative approach}

In this section, we show Theorem \ref{th:local without Rnew} on $\dint\Gamma_{yx}^1$. More precisely, we explain why the $\dint\Gamma$ map introduced in \cite{BN23} is the same as in the original work \cite{LOTP}, thanks to a characterisation used therein. We begin with the easy lemma.
 \begin{lemma}\label{commutmult}
 We have the commutation property
 \begin{equation}
 \d^n\Pi^{0,\times}_x=\Pi^{0,\times}_x\mathcal{D}_n.
 \end{equation}
 \end{lemma}
 
 \begin{proof}
 The statement is obvious on elementary trees, and we therefore omit the proof. For a tree of the product form $\sigma\tau$, we have by Leibniz rule
 \begin{equs}
 \d^n(\Pi^{0,\times}_x\sigma\tau)(y)&=\d^n\left((\Pi^{0,\times}_x\sigma)(y)\times(\Pi^{0,\times}_x\tau)(y)\right)\\
 &=\sum_{k=0}^n\binom{n}{k}\d^k(\Pi^{0,\times}_x\sigma)(y)\times\d^{n-k}(\Pi^{0,\times}_x\tau)(y)\\
 &=\sum_{k=0}^n\binom{n}{k}(\Pi^{0,\times}_x\mcD_k\sigma)(y)\times(\Pi^{0,\times}_x\mcD_{n-k}\tau)(y)\\
 &=(\Pi^{0,\times}_x\mcD_n(\sigma\tau))(y).
 \end{equs}
 For a tree of the form $\mcI_a(\tau)$, we get
 \begin{equs}
~&\d^n(\Pi^{0,\times}\mcI_a\tau)(y)\\
&=(\d^{a+n} K*\Pi_{x}^{0}\tau)(y)
 -\sum_{\tiny{|k|_{\s}\leq \deg_0(\mcI_a(\tau))}}\frac{(y-x)^{k-n}}{(k-n)!}(\d^{a+k} K*\Pi_{x}^{0}\tau)(x),\\
 &=(\d^{a+n} K*\Pi_{x}^{0}\tau)(y)
 -\sum_{\tiny{|k|_{\s}\leq \deg_0(\mcI_{a+n}(\tau))}}\frac{(y-x)^{k}}{k!}(\d^{a+k+n} K*\Pi_{x}^{0}\tau)(x),\\
 &=(\Pi^{0,\times}\mcI_{a+n}\tau)(y)=(\Pi^{0,\times}\d^n\mcI_a\tau)(y).
 \end{equs}
 We have used a change of index in the sum ans the fact that $\deg_0(\mcI_{a+n}(\tau))= \deg_0(\mcI_{a}(\tau))-|n|_{\s}$.
 \end{proof}
 

We can now turn to the main result of this section.
 \begin{theorem}\label{th:local without Rnew} For every $\tau\in\mathcal{T}_0 $ satisfying Assumption~\ref{assumpt1}, and $R$ satisfying Assumption~\ref{assumpt3} one has for $|n|_{\s} \leq N$
\begin{equs} \label{local without R}
\partial^n \left( \Pi^{0,\times}_y  \dint \Gamma_{yx}^1  \tau-\delta\Pi^{0,\times}_x \tau \right)(y)=0,
\end{equs}
and hence one has for  $\tau\in\mathcal{T}_{0}$
\begin{equs} \label{local R}
\partial^n\left(  \Pi^{0}_y  \dint \Gamma_{yx}^1  \tau -\delta \Pi_x^{0} \tau \right)(y)=0.
\end{equs}
\end{theorem}

\begin{proof}
To prove the result, we follow closely the lines of the proof of Theorem 4.5 in \cite{BN23}. The only non-trivial modification is to prove \eqref{local without R} for a tree of the form $\mcI_a(\tau)$. It is indeed the only part where the fact that the base point and the evaluation point are taken to be the same on the left-hand side matters. We recall the following sequence of equalities
\begin{equs}
	\label{computation}
	\begin{aligned}
\d^n\left( \Pi^{0,\times}_y  \dint \Gamma_{yx}^1  \mathcal{I}_a\tau \right)(y) & = \d^n\left( \Pi^{0,\times}_y \mathcal{Q}_0 \left( \Gamma_{yx}^{1} \otimes f_x^{1} \right) \hat{\Delta}_0 \mathcal{I}_a(D_{\Xi} \tau) \right)(y)
\\ & = \d^n\left( \left( \Pi^{1,\times}_y \Gamma_{yx}^{1} \cdot \right)(y) \otimes f_x^{1} \right) \hat{\Delta}_1 \mathcal{I}_a(D_{\Xi} \tau)
\\ & =\d^n\left( \left(  \Pi^{1,\times}_x  \cdot \right)(y) \otimes f_x^{1} \right) \hat{\Delta}_1 \mathcal{I}_a(D_{\Xi} \tau)
\\ & = \d^n\left(\Pi^{0,\times}_x \mathcal{I}_a(D_{\Xi} \tau) \right) (y)
\\ & = \d^n\left(\delta \Pi^{0,\times} \mathcal{I}_a(\tau) \right)(y),
\end{aligned}
\end{equs}
The crucial point is the identity
\begin{equs}
	\d^n(\Pi^{0,\times}_y \mathcal{Q}_0  \Gamma_{yx}^{1} \mathcal{I}_a(D_{\Xi} \tau))(y) = \d^n(\Pi^{1,\times}_y   \Gamma_{yx}^{1} \mathcal{I}_a(D_{\Xi} \tau))(y).
	\end{equs}
This due to the fact that, for $|n|\leq N$,
\begin{equs}
	\deg_1( \mathcal{I}_{a+n}(D_\Xi \tau') ) > 0,
\end{equs}
for every $ \tau' $ subtree of $ \tau $ including the root of $ \tau $ (see Assumption \ref{assumpt3}) and, using Lemma \ref{commutmult},
\begin{equs}
	(\d^n\Pi^{1,\times}_y  \mathcal{I}_a(D_{\Xi} \tau'))(y)&=(\Pi^{1,\times}_y \mcD_n \mathcal{I}_a(D_{\Xi} \tau'))(y)\\
	&=(\Pi^{1,\times}_y  \mathcal{I}_{a+n}(D_{\Xi} \tau'))(y)\\
	&=0.
\end{equs}
The proof of $\eqref{local R}$ is verbatim the proof in \cite{BN23}.
\end{proof}

To conclude the link with \cite{LOTP}, we need to twist slightly the framework, since the authors are working with duality. They use a dual model, defined by
\begin{equation}
\Pi_x^0(y)=\sum_{\tau}\frac{(\Pi_x^0\tau)(y)}{S(\tau)}\tau,
\end{equation}
and they use $\dint\Gamma_{yx}^{1,*}$, which is the formal adjoint of $\dint\Gamma_{yx}^1$ for the scalar product $\langle\cdot,\cdot\rangle$. 

\begin{proposition}\label{propeqdual}
We have the equivalence, for $|n|_{\s}\leq N$,
\begin{equation}
\d^n(\delta\Pi_x^0-\dint\Gamma^{1,*}_{xy}\Pi^0_y)(y)=0\Leftrightarrow\forall\tau\in\CT_0,~\partial^n \left(\delta \Pi_x^{0} \tau - \Pi^{0}_y  \dint \Gamma_{yx}^1  \tau\right)(y)=0.
\end{equation}
\end{proposition}

\begin{proof}
We have first
\begin{equs}
\langle\delta\Pi^0_x,\sigma\rangle&=\left\langle\sum_\tau\frac{\delta\Pi^0_x\tau}{S(\tau)}\tau,\sigma\right\rangle=\sum_\tau\frac{\delta\Pi^0_x\tau}{S(\tau)}\langle\tau,\sigma\rangle\\
&=\delta\Pi^0_x\sigma,
\end{equs}
and
\begin{equs}
\langle\dint\Gamma^{1,*}_{xy}\Pi^0_y,\sigma\rangle&=\left\langle\sum_\tau\frac{\Pi^{0}_y \tau}{S(\tau)}\dint\Gamma^{1,*}_{xy}\tau,\sigma\right\rangle=\sum_\tau\frac{\Pi^{0}_y \tau}{S(\tau)}\langle\dint\Gamma^{1,*}_{xy}\tau,\sigma\rangle\\
&=\sum_\tau\frac{\Pi^{0}_y \tau}{S(\tau)}\langle\tau,\dint\Gamma_{xy}^1\sigma\rangle=\Pi^{0}_y  \dint \Gamma_{yx}^1  \sigma.
\end{equs}
Taking $\d^n$ on both sides allows to conclude.
\end{proof}
This proposition is exactly what we want, since in \cite{BOT25,LOTP}, the application $\dint\Gamma^{1,*}_{yx}$ is characterized on multi-indices by the left-hand side of Proposition \ref{propeqdual}. See for example in \cite[Section 5.4]{LOTP} or  the lecture notes \cite[Section 2.7]{BOT25}. Note that the identity that we proved stands on trees, but this is far from being a problem, as every multi-index can be written as a linear combination of trees.
 
 \subsection{Recursive definition}
 
 \label{Sec::4}
 
In this section, we show that the modelled distribution $\Tilde{f}^\tau_x$ used in \cite{HS23} can be written in a form that is close to the map $\dint\Gamma$ defined in \cite{BN23}, namely $\dint\Gamma^2$ defined in \ref{maindef}. See also the definition in the lecture notes \cite{BH25}, p.17. 
We are ready to state the main theorem of this section.
\begin{theorem} \label{Thm_Martin} The map
$\dint\Gamma_{yx}^2$ satisfies the following induction.
\begin{equation}
\begin{aligned}
&\dint\Gamma_{yx}^2\Xi=\dot\Xi+\delta\xi(y),\quad\dint\Gamma_{yx}^2X^k=0,\\
&\dint\Gamma_{yx}^2(\sigma\tau)=\dint\Gamma_{yx}^2\tau\times\Gamma^2_{yx}\sigma+\Gamma^2_{yx}\tau\times\dint\Gamma_{yx}^2\sigma,\\
&\dint\Gamma_{yx}^2\mcI_a(\tau)=(\mcI_a+\mcJ_{a}^2(y))(\dint\Gamma_{yx}^2\tau)-\Gamma_{xy}^2P_{\deg_0(\mcI_a\tau)}\big(\mcJ_{a}^2(x)(\dint\Gamma_{xx}^2\tau)\big).
\end{aligned}
\end{equation}
$P_\eta$ is the canonical projection onto trees of homogeneity less than $\eta$.
\end{theorem}

\begin{proof}
Let us start with the initialization. We have that $\hat\Delta_2\dot\Xi=\dot\Xi\otimes\one-\one\otimes\dot\Xi$, so that 
\begin{equs}
\dint\Gamma_{yx}^2\Xi&=\Gamma_{yx}^2\dot\Xi+\delta\xi(x)\one\\
&=\dot\Xi+(\delta\xi(y)-\delta\xi(x))\one+\delta\xi(x)\one\\
&=\dot\Xi+\delta\xi(y)\one.
\end{equs}
The identity on $X^k$ is trivial since $D_\Xi X^k=0$. Let us move to the product identity. First note that for $\tau\in T_0$, $\hat\Delta_2\tau=\tau\otimes\one$. We get
\begin{equs}
\hat\Delta_2D_\Xi(\sigma\tau)&=\hat\Delta_2\left(D_\Xi\sigma\times\tau+\sigma\times D_\Xi\tau\right)\\
&=\hat\Delta_2(D_\Xi\sigma)\times\hat\Delta_2\tau+\hat\Delta_2\sigma\times \hat\Delta_2D_\Xi\tau\\
&=\hat\Delta_2(D_\Xi\sigma)\times(\tau\otimes\one)+(\sigma\otimes\one)\times \hat\Delta_2D_\Xi\tau.
\end{equs}
From this we get the desired result. The last identity on planted trees is the most thorough to get. The backbone of the proof is formula \eqref{backbone} that we recall below 
\begin{equation*}
\Gamma^2_{yx}(\mcI_a+\mcJ_{a}^2(x))=(\mcI_a+\mcJ_{a}^2(y))\Gamma_{yx}^2.
\end{equation*}
We start with
\begin{equation*}
\begin{aligned}
\left( \Gamma_{yx}^{2} \otimes f_x^{2} \right) \hat{\Delta}_2 \CI_a\tau &=  ( \Gamma_{yx}^{2}\mcI_a\otimes f_x^{0})\hat{\Delta}_{2} \tau \nonumber\\
&~~~~- \sum_{\deg_0(\mcI_a(\tau)) \leq \vert\ell\vert_{\s}}  \Gamma_{yx}^{2}  \frac{X^\ell}{\ell!}\otimes \big( f_x^{2} \mcI^{+,2}_{a+\ell}\otimes f_x^{2} \big)\hat{\Delta}_{2} \tau.
\end{aligned}
\end{equation*}
We write in Sweedler's notation $\hat\Delta_2\tau=\sum_{(\tau)}\tau_1\otimes\tau_2$. We treat the first term in the right-hand side of the above equation. We get, using \eqref{backbone},
\begin{equs}
(\Gamma_{yx}^{2}\mcI_a\otimes f_x^{2})(\tau_1\otimes\tau_2)&= \Gamma_{yx}^{2}\mcI_a\tau_1 f_x^{2}(\tau_2)\\
 &=\left((\mcI_a+\mcJ_{a}^2(y))\Gamma_{yx}^2\tau_1-\mcJ_{a}^2(x)\tau_1\right) f_x^{2}(\tau_2),
\end{equs}
which gives
\begin{equs}
(\Gamma_{yx}^{2}\mcI_a\otimes f_x^{2})\hat{\Delta}_{2} D_\Xi\tau=(\mcI_a+\mcJ_{a}^2(y))\dint\Gamma_{yx}^2\tau-\Gamma_{yx}^2\mcJ_{a}^2(x)\dint\Gamma_{xx}^2.
\end{equs}
We recall that $\dint\Gamma_{xx}^2=(\mathrm{id}\otimes f_x^2)\hat\Delta_2D_\Xi$. The other term gives 
\begin{equs}
~& \Gamma_{yx}^{2}\sum_{\deg_0(\mcI_a(\tau)) \leq \vert\ell\vert_{\s}}   \frac{X^\ell}{\ell!}\otimes \big( f_x^{2} \mcI^{+,2}_{a+\ell}\otimes f_x^{2} \big)(\tau_1\otimes\tau_2)\\
&=\Gamma_{yx}^{2}\sum_{\deg_0(\mcI_a(\tau)) \leq \vert\ell\vert_{\s}}   \big( f_x^{2}( \mcI^{+,2}_{a+\ell}\tau_1) f_x^{2}(\tau_2) \big) \frac{X^\ell}{\ell!}\\
&=-\Gamma_{yx}^{2}\sum_{\deg_0(\mcI_a(\tau)) \leq \vert\ell\vert_{\s}}  (\d^{a+\ell}K*\Pi_x^2\tau_1)(x) \frac{X^\ell}{\ell!} f_x^{2}(\tau_2).
\end{equs}
Putting everything together, we have
\begin{equs}
\left( \Gamma_{yx}^{2} \otimes f_x^{2} \right) \hat{\Delta}_2 \CI_a\tau=(\mcI_a+\mcJ_{a}^2(y))\dint\Gamma_{yx}^2\tau-\Gamma_{yx}^2P_{\deg_0(\mcI_a\tau)}\mcJ_{a}^2(x)\dint\Gamma_{xx}^2.
\end{equs}

\end{proof}

We now recall that the inductive definition stated in Theorem \ref{Thm_Martin} is what is used to define the modelled distribution $\Tilde{f}^\tau_x$ in \cite{HS23}, thus allowing to conclude.

\end{document}